\documentclass[12pt]{amsart}
\usepackage{amsmath,amsthm,amsfonts,amssymb,latexsym,verbatim,bbm}
\input xy
\xyoption{all}
\usepackage{hyperref,multirow}
\usepackage{tikz}

\allowdisplaybreaks[2] \textwidth15.1cm \textheight22cm \headheight12pt \oddsidemargin.4cm
\theoremstyle{plain}
\newtheorem{theorem}{Theorem}[section]
\newtheorem{thm}[theorem]{Theorem}
\newtheorem{lemma}[theorem]{Lemma}

\newtheorem{prop}[theorem]{Proposition}
\newtheorem{example}[theorem]{Example}
\newtheorem{definition}[theorem]{Definition}
\newtheorem{defn}[theorem]{Definition}

\newtheorem{cor}[theorem]{Corollary}

\newcommand{\Z}{\mathbb{Z}}

\DeclareMathOperator{\pred}{\mathsf{D}}
\DeclareMathOperator{\supp}{\mathsf{S}}
\DeclareMathOperator{\Tria}{HQ}

\newcommand\TabTop{\rule{0pt}{2.6ex}}
\newcommand\TabBot{\rule[-1.2ex]{0pt}{0pt}}

\begin{document}

\title[Coxeter groups and triangle avoidance]{Rationally smooth elements of Coxeter groups and triangle group avoidance}

\author{Edward Richmond}
\address{Department of Mathematics, University of British Columbia, Vancouver, BC V6T 1Z2, Canada}
\email{erichmond@math.ubc.ca}

\author{William Slofstra}
\address{Department of Mathematics, University of British Columbia, Vancouver, BC V6T 1Z2, Canada}
\email{slofstra@math.berkeley.edu}

\begin{abstract}
We study a family of infinite-type Coxeter groups defined by the avoidance of certain rank $3$
parabolic subgroups. For this family, rationally smooth elements can be detected by looking at only
a few coefficients of the Poincar\'{e} polynomial.  We also prove a factorization theorem for the
Poincar\'{e} polynomial of rationally smooth elements. As an application, we show that a large
class of infinite-type Coxeter groups have only finitely many rationally smooth elements.  Explicit
enumerations and descriptions of these elements are given in special cases.
\end{abstract}

\maketitle

\section{Introduction}

Let $W$ be a Coxeter group with finite generating reflection set $S$, and let $\ell$ and $\leq$
denote the length function and Bruhat order on $W$, respectively. Let $e\in W$
denote the identity of $W.$  By definition, $W$ is the group generated by $S$
satisfying relations $(st)^{m_{st}} = e$, where $m_{st} \in
\{1,2,3,\ldots,\infty\}$ such that $m_{st} = 1$ if and only if $s=t$.  If $m_{st} = \infty$, then
by convention the relation $(st)^{\infty} = e$ is omitted.  The Poincar\'{e} series
$$P_w(q) = \sum_{x \leq w} q^{\ell(x)}$$ of an element $w \in W$ is a polynomial of degree
$\ell(w)$. An element $w$ is said to be \emph{palindromic} (or \emph{rationally smooth}) if the coefficients of
$P_w(q)$ are the same whether read from top degree to bottom degree, or in reverse.\footnote{The
term rationally smooth seems to be more common in the literature; we use the term palindromic
to be inclusive of the non-crystallographic case.} In other words, if we write $P_w(q) = \sum a_i
q^i$, then $w$ is palindromic when $a_i = a_{\ell(w)-i}$ for all $i$.


An important question in the combinatorics of Coxeter groups is to describe the set of palindromic
elements of $W$. This question stems from its connection with the topology of Schubert varieties. A
Coxeter group is \emph{crystallographic} if $m_{st} \in \{2,3,4,6,\infty\}$ for all $s \neq t$. If
$W$ is crystallographic, then it can be realized as the Weyl group of a Kac-Moody algebra. The
Schubert subvarieties of the full flag variety corresponding to this algebra are indexed by the
elements of $W$. Carrell and Peterson prove that the Schubert variety indexed by $w$ is rationally
smooth if and only if $w$ is palindromic \cite{Ca94}.  Furthermore, $w$ is palindromic if and only
if the Kazhdan-Lusztig polynomial indexed by $(x,w)$ is equal to $1$ for all $x \leq w$
\cite{KL79,KL80}. If $W$ is crystallographic, then it is sufficient that the Kazhdan-Lusztig
polynomial indexed by $(e,w)$ be equal to $1$ \cite{Ca94}.  For
Schubert varieties of simply laced types $A,D$ and $E,$ the notion of smooth and rationally smooth are
equivalent. For finite Weyl groups, the palindromic elements are well understood.  In particular,
they can be characterized using permutation pattern-avoidance in classical types $A,B,C,$ and $D$ and using
root system avoidance in all types \cite{BL00,BP05,Bi98,LS90}. The characterization using
permutation pattern avoidance has recently been extended to the affine type $A$ case as well
\cite{BC10}. The generating series for the number of palindromic elements in $A_n$, as $n$ varies,
is also known \cite{BL00, Bo98, St94}.


While the theory of palindromic elements is well-developed for finite and affine Coxeter groups, the situation for general Coxeter groups is
quite different. In particular, it seems to be quite difficult to determine whether or not an
element of a general Coxeter group is rationally smooth. In this paper, we introduce a family of
Coxeter groups (mostly) outside the finite and affine cases, for which it is possible to determine
if an element is rationally smooth by looking at just a few coefficients of the Poincar\'{e}
polynomial. The family in question is defined as the set of all Coxeter groups which do not contain
certain triangle groups as standard parabolic subgroups. A \emph{triangle group} is a Coxeter group with
$|S|=3$. Triangle groups arise naturally in arithmetic geometry and the study of tessellations of
triangles on Riemann surfaces, see e.g. \cite{Be83}.  We will denote a triangle group by the triple $(m_{rs},m_{rt},m_{st})$ where  $S=\{r,s,t\}.$
We say a Coxeter group $W$ \emph{contains the triangle} $(a,b,c)$ if there exists a subset $\{r,s,t\}\subseteq S$ such that $(a,b,c)=(m_{rs},m_{rt},m_{st}).$  If $S$
contains no such subset, then we say $W$ \emph{avoids the triangle} $(a,b,c)$.  We are in interested in
the groups which avoid the following special set of triangle groups:
$$\Tria:=\{(2,b,c)\ |\ b,c\geq 3\ \text{and}\ b<\infty\}$$
The set $\Tria$ (Hecke quotients) is the set of quotients of the Hecke triangle group
$(2,p,\infty)$, $p \geq 3$, which is a generalization of the well-known modular
group $(2,3,\infty)$.  Every finite Coxeter group of rank $\geq 3$ contains a triangle in $\Tria$, and
the same is true of affine Coxeter groups, with the exception of $(3,3,3)$,
which is the affine group $\tilde{A}_2$. However, there are many
crystallographic Coxeter groups which do avoid $\Tria$; for example, any
Coxeter group with no commuting relations (i.e. $m_{st} \geq 3$ for all $s \neq
t$) avoids $\Tria$.  Any Coxeter group defined by only by commuting and infinite relations also avoids $\Tria$.


To state our main theorem, we make the following definition:
\begin{defn}
    Let $w$ be an element of a Coxeter group $W$, and write $P_w(q) = \sum a_i
    q^i$ for the Poincar\'{e} polynomial of $w$.  We say that $w$ is \emph{$k$-palindromic} if $a_{i} = a_{\ell(w) -
    i}$ for all $0 \leq i < k$.
\end{defn}
Note that if $k=\infty$, then we recover the usual notion of palindromic
elements, and that every element is $1$-palindromic with $a_0=a_{\ell(w)}=1.$
If $W$ is crystallographic, then $k$-palindromicity can be detected from
the Kazhdan-Lusztig polynomial. Let $T_{e,w} = 1
+ \sum_{i \geq 0} b_i q^i$ be the Kazhdan-Lusztig polynomial indexed by
$(e,w)$. A theorem of Bjorner and Ekedahl states that, for crystallographic
groups, an element $w \in W$ is $k$-palindromic if and only if $b_i = 0$ for $0
\leq i < k$ \cite{BE09} (note that $b_0 = 0$ always).


We now state the main theorem:
\begin{theorem}\label{T:main} Let $W$ be a Coxeter group which avoids all triangle groups in $\Tria.$ Then every 4-palindromic $w\in W$ is palindromic.

\smallskip

Furthermore, if $W$ avoids all triangle groups $(3,3,c)$ where $3<c<\infty$, then every $2$-palindromic $w\in W$ is palindromic.
\end{theorem}

Given a Coxeter group, it is natural to ask whether there is a number $k$ such that every
$k$-palindromic element is palindromic. This question appears to be open in general.  Billey
and Postnikov have conjectured that if $W$ is a finite simply-laced Weyl group with $n$ generators,
then every $(n+1)$-palindromic element of $W$ is palindromic \cite{BP05}.  In type $A_n$, it is known
that every $(n-1)$-palindromic element is palindromic \cite{BP05}.

The proof of Theorem \ref{T:main} is based on a factorization theorem for the Poincar\'{e}
polynomial of $2$-palindromic elements in Coxeter groups which avoid $\Tria$.   In the classical
groups of finite type $A,B,C,$ and $D$, it is known that the Poincar\'{e} polynomial of a rationally smooth
element factors into a product of $q$-integers (see Equation \eqref{E:qinteger}) \cite{Bi98,Ga98}.
In fact, it is possible to see this factorization combinatorially, writing each palindromic element
$w$ as a reduced product $w_1 \cdots w_{|S|}$, such that each $q$-integer factor of the
Poincar\'{e} polynomial equals the (relative) Poincar\'{e} polynomial of the $w_i$'s. We prove a
similar result for $2$-palindromic elements in Coxeter groups which avoid $\Tria$.  This result has
a number of applications.  For example, we show there are many infinitely large Coxeter groups
with only a finite number of palindromic elements.  We also give explicit descriptions of palindromic elements in special cases.
In the case of uniform Coxeter groups $W(m,n),$ defined by $m_{st}=m$ for all $s\neq t$ and
$|S|=n,$ we calculate the generating series for the number of palindromic elements weighted by
length. Formulas for these generating series are stated in Propositions \ref{P:Uniform1} and
\ref{P:Uniform2}. We also observe that the $\Tria$-avoiding groups form the largest class of
Coxeter groups for which our factorization theorem can hold.

\subsection{Organization} Section \ref{S:background} contains some background
material and elementary lemmas used to state the factorization theorem. Section \ref{S:mainresults}
states the main factorization theorem and its consequences, including the proof of Theorem
\ref{T:main} and enumerative results.  In Section \ref{S:HQgroups}, we consider triangle groups in the set $\Tria$
and prove the main results cannot hold for any Coxeter group containing these triangle groups.
Section \ref{S:descents} gives some elementary lemmas on the descent sets of Coxeter
groups avoiding $\Tria$. Finally, Section \ref{S:Proof} proves the main factorization theorem.

\section{Background and terminology}\label{S:background}

Let $W$ be a Coxeter group with simple generator set $S.$  For basic facts on Coxeter groups, we
refer the reader to \cite{BB05}.  Let $\ell(w)$ denote the length of $w\in W.$  We say $w=uv\in W$
is a \emph{reduced factorization} if $\ell(w)=\ell(u)+\ell(v).$  A special type of reduced
factorization can be constructed from any subset $J \subseteq S.$  Let $W_J$ denote the standard
parabolic subgroup of $W$ generated by $J.$  Let $W^J$ denote the set of minimal length coset
representatives of $W_J\backslash W.$  Every element $w \in W$ can be written uniquely as $w = uv$
where $u \in W_J$, $v \in W^J$ and $\ell(w)=\ell(u)+\ell(v).$   We call this reduced factorization
of $w$ the \emph{parabolic decomposition} with respect to $J$.


Let $\leq$ denote the Bruhat order on $W.$  If $u\leq v \in W$, then the interval $[u,v]$ denotes
the set of elements $x \in W$ such that $u \leq x \leq v$.  For any $w\in W$ we can define the
Poincar\'{e} polynomial

$$P_w(q):= \sum_{x \in [e,w]} q^{\ell(x)}.$$

The Poincar\'{e} polynomial relative to $J\subseteq S$ of an element $w \in W$ is defined to be

$$P_w^J(q) := \sum_{x \in [e,w] \cap W^J} q^{\ell(x)}.$$

If $w \in W^J$, then $P_w^J(q)$ is a polynomial of degree $\ell(w).$  If $J=\emptyset$, then
$P_w^J(q)=P_w(q)$.  Recall that for any $J,$ the poset $[e,w] \cap W_J$ has a unique maximal
element. The following proposition is due to Billey and Postnikov in \cite[Theorem 6.4]{BP05}.
\begin{prop}[\cite{BP05}]\label{P:para}
   Let $J\subseteq S$ and let $w=uv$ be a parabolic decomposition with respect to $J.$  Then $u$ is the unique maximal element of $[e,w] \cap W_J$ if and only if
    $$P_w(q) = P_u(q) \cdot P_v^J(q).$$
\end{prop}


While the proof of Proposition \ref{P:para} given in \cite{BP05} is stated only for finite Weyl
groups, it easily extends to all Coxeter groups.  A parabolic decomposition $w=uv$ is called
a \emph{BP-decomposition} of $w$ if $u$ is the unique maximal element of $[e,w] \cap W_J$.


For any $w\in W,$ define the sets
\begin{align*}
\supp(w)&:=\{u\leq w\ |\ \ell(u)=1\}\\
\pred(w)&:=\{u\leq w\ |\ \ell(u)=\ell(w)-1\}\\
D_R(w)&:=\{s\in S\ |\ \ell(ws)<\ell(w)\}\\
D_L(w)&:=\{s\in S\ |\ \ell(sw)<\ell(w)\}.
\end{align*}
The sets $\supp(w)$ and $\pred(w)$ are known as the support and divisor sets of $w.$  The sets $D_R(w)$ and $D_L(w)$ are called the \emph{right} and \emph{left descent sets} of $w$ respectively and are contained in $\supp(w)$.  We use these sets to give an equivalent characterization of a BP-decomposition.

\begin{lemma}\label{L:BPlemma}
A parabolic decomposition $w=uv$ is a BP-decomposition if and only if $\supp(v)\cap J\subseteq D_R(u).$
\end{lemma}

\begin{proof}If $w=uv$ is a BP-decomposition, then $u$ is the unique longest element of $[e,w] \cap W_J$.
If there exists $x\in \supp(v)\cap J$ and $x\notin D_R(u)$, then $\ell(ux)=\ell(u)+1$ and
$ux\in[e,w] \cap W_J$ which is a contradiction.


Conversely, assume that $\supp(v)\cap J\subseteq D_R(u)$ and let $\bar u$ denote the maximal
element in $[e,w] \cap W_J.$  Since $\bar u$ is unique, we have that $u\leq\bar u.$   We now show
that $\bar u \leq u.$  Let $$\bar u=u'v'$$ be a reduced factorization which maximizes $\ell(u')$
under the conditions that $u'\leq u$ and $v'\leq v.$  Suppose that $v'\neq e.$  Then there exists
$y\in D_L(v')\backslash D_R(u').$  By assumption, we have that $y\in D_R(u).$
Taking a reduced decomposition for $u$ with $y$ appearing at the end, we see that $u'\leq uy,$ and hence $u'$ can be extended, a contradiction.


\end{proof}

We remark that one direction of Lemma \ref{L:BPlemma} is proved in \cite[Lemma 10]{OY10}.  Another property of BP-decompositions is the following lemma.

\begin{lemma}\label{L:BPassoc} Let $J_1\subseteq J_2\subseteq S$ and let $v_1v_2v_3$ be a reduced factorization such that $v_1v_2$ and $(v_1v_2)(v_3)$ are BP-decompositions with respect to $J_1$
and $J_2$ respectively.  Then $v_1(v_2v_3)$ is a BP-decomposition with respect to $J_1$.  \end{lemma}

\begin{proof}By definition, we have $v_1v_2$ is maximal in $[e, v_1v_2v_3]\cap W_{J_2}.$  In particular, if $u$ denotes the maximal element in $[e,v_1v_2v_3]\cap W_{J_1},$
then $u\leq v_1v_2$ since $W_{J_1}\subseteq W_{J_2}.$  But now $u$ is maximal in $[e,v_1v_2]\cap W_{J_1},$ which implies that $u=v_1.$  \end{proof}

Clearly, if $P_w(q)=\sum a_i q^i,$ then $|\supp(w)|=a_{1}$ and $|\pred(w)|=a_{\ell(w)-1}.$  We now consider a special class of parabolic decompositions.

\begin{defn}
We say that $w=uv$, a parabolic decomposition with respect to $J$, is a Grassmannian factorization if
$J=\supp(u)$ and $|\supp(w)|=|\supp(u)|+1$
\end{defn}


It is easy to see that every element $w \in W$ of length $\geq 2$ has a
Grassmannian factorization. The term ``Grassmannian" comes from the fact that $v$ is a Grassmannian element
of $W$ which, by definition, has $|D_L(v)|=1.$  Note that a Grassmannian factorization is not necessarily a BP-decomposition.  Although elementary, this concept is quite useful.
For example, we can use it to prove:

\begin{lemma}\label{L:predgeqsupp}
    $|\pred(w)| \geq |\supp(w)|$.
\end{lemma}

\begin{proof}
We proceed by induction on $\ell(w)$. The proposition is true if $\ell(w) = 1$, so suppose $\ell(w)
\geq 2$.  Let $w=uv$ be a Grassmannian factorization with respect to $J.$  By induction,
$|\pred(u)| \geq |\supp(u)|$.


If $u' \in \pred(u)$, then $u' v \in \pred(w)$, since $v \in W^J$. Now $v$ is not the identity, so
we can write $v' = v s\in W^J$ with $s\in S$ and $\ell(v')=\ell(v)-1.$  Consequently $u v' \in
\pred(w)$. Moreover, $uv'\neq u'v$ for any $u'\in\pred(u)$ since they are both parabolic
decompositions with respect to $J$ and $u\neq u'.$    Hence
\begin{equation}\label{E:predgeqsupp}|\pred(w)|\geq
|\pred(u)|+1\geq |\supp(u)|+1=|\supp(w)|.\end{equation} This completes the proof.
\end{proof}
We remark that, for crystallographic Coxeter groups, Bjorner and Ekedahl prove a much
stronger version of Lemma \ref{L:predgeqsupp} concerning all the coefficients of $P_w(q)$ \cite[Theorem A]{BE09}.


We can continue to decompose any Grassmannian factorization $w=uv$ by taking a Grassmannian factorization
of $u$. We say that
$$w=v_1v_2\cdots v_{|\supp(w)|}$$ is a \emph{complete Grassmannian factorization} of $w$ if for
every $i<|\supp(w)|$, we have that $(v_1\cdots v_i)(v_{i+1})$ is a Grassmannian factorization.  Observe that if each $(v_1\cdots v_i)(v_{i+1})$ is also a BP-decomposition, then by Lemma \ref{L:BPassoc}, we have $(v_1\cdots v_i)(v_{i+1}\cdots v_k)$ is BP decomposition for any $i<k\leq |\supp(w)|.$


By definition, $w$ is $2$-palindromic if and only if $|\pred(w)|=|\supp(w)|.$  The following
lemma gives an inductive characterization of the $2$-palindromic property.

\begin{lemma}\label{L:2palcondition}
Suppose that $w=uv$ is a Grassmannian factorization.  Then $w$ is $2$-palindromic if and only if $u$
is $2$-palindromic and $|u\cdot\pred(v)\cap\pred(w)|=1.$
\end{lemma}

\begin{proof}
Equality holds in Equation \eqref{E:predgeqsupp} if and only if $|\pred(u)|=|\supp(u)|$ and $u\cdot\pred(v)\cap\pred(w)=\{uvs\}$ where $s\in D_R(v).$  \end{proof}


\section{The factorization theorem}\label{S:mainresults}

The main technical theorem of this paper is the following:

\begin{thm}\label{T:mainchar}
Suppose that $W$ avoids all triangle groups in $\Tria$.  Let $w\in W$ be $2$-palindromic and fix a
Grassmannian factorization $w=uv$ with respect to $J\subseteq S.$  Then $w=uv$ is a
BP-decomposition with respect to $J$ such that $|\supp(v)|\leq 3.$

\smallskip

Moreover, if $|\supp(v)|= 3$ and $\supp(v)=\{r,s,t\},$ then one of the following is true:

\smallskip

\begin{enumerate}
\item $v=trv'$ with $v'=\underbrace{stst\ldots}_{m_{st}-1}$ where $\supp(v)$ generates the triangle group $(3,m_{rs},m_{st})$ with $m_{rt}=3$ and $3 \leq m_{st}<\infty$, $3 \leq m_{rs} \leq \infty$.

\smallskip

\item $v=rstrv'$ with $v'=\underbrace{stst\ldots}_{m_{st}-1}$ where $\supp(v)$ generates the triangle group $(3,3,m_{st})$ with $3 < m_{st}<\infty.$

\smallskip

\item $v=strstr\cdots$ is a spiral word\footnote{A spiral word is word which cycles through a set of generators in a fixed order.} of even length where $\supp(v)$ generates the triangle group $(3,3,3).$
\end{enumerate}\end{thm}

Theorem \ref{T:mainchar} says that if $W$ avoids triangle groups in $\Tria$, then the Poincar\'{e} polynomial $P_w(q)$ of a 2-palindromic element $w\in W$ factors along any Grassmannian factorization of $w=uv$.  Moreover, the possibilities for the factor $P^J_v(q)$ is limited by the fact that $|\supp(v)|\leq 3.$  Note that parts (1) and (3) of the theorem overlap when $m_{rs}=m_{st}=3.$  The proof of this theorem is the focus of Section
\ref{S:Proof}.  The remainder of this section is devoted to consequences of Theorem \ref{T:mainchar}.


Fix a $2$-palindromic element $w\in W$ and a Grassmannian factorization $w=uv$ with respect to
$J\subseteq S.$  Theorem \ref{T:mainchar} can be used, together with Lemma \ref{L:2palcondition}, to completely determine the polynomial $P_w(q).$  By Theorem \ref{T:mainchar} and
Proposition \ref{P:para}, we have that $$P_w(q) = P_u(q) \cdot P_v^J(q),$$  so it suffices to characterize all possible polynomials $P_v^J(q).$  For any integer
$k\geq 1$ define the $q$-integer
\begin{equation}\label{E:qinteger}[k]_q:=1+q\cdots +q^{k-1}.\end{equation}  If $|\supp(v)|\leq 2,$ then any $v'\leq v$ where $v'\in W^J$ is given by a prefix of the unique reduced word of $v.$  This implies
\begin{equation}\label{E:rank2poly}P_v^J(q)=[\ell(v)+1]_q.\end{equation}

If $|\supp(v)|=3$, it suffices to compute $P_v^J(q)$ in all the cases of Theorem \ref{T:mainchar}.  We have the following lemma.

\begin{lemma}\label{L:rank3v}Suppose we have $w=uv$ as in Theorem \ref{T:mainchar} with $|\supp(v)|=3.$  Then the following are true:  \begin{enumerate}
\item If $v$ satisfies the conditions in Theorem \ref{T:mainchar} part (1), then
$$P_v^J(q)=[\ell(v)+1]_q+q^2[\ell(v)-3]_q.$$
\item If $v$ satisfies the conditions in Theorem \ref{T:mainchar} part (2), then
$$P_v^J(q)=[\ell(v)+1]_q+q^2[\ell(v)-3]_q+q^4[\ell(v)-6]_q.$$
\item If $v$ satisfies the conditions in Theorem \ref{T:mainchar} part (3) with $\displaystyle k=\left\lfloor\frac{\ell(v)}{4}\right\rfloor$, then
$$P_v^J(q)=\sum_{i=0}^k q^{2i}[\ell(v)-4i+1]_q.$$
\end{enumerate}\end{lemma}

\begin{proof}
Part (3) is proved in \cite[Proposition 2.4]{Mi86} where certain Poincar\'{e} polynomials of Schubert varieties in
the affine Grassmannian of type $A$ are calculated.  Parts (1) and (2) can be deduced from elementary
counting arguments of the sets $$\{v'\in W^J\cap[e,v]\ |\ \ell(v')=i\}.$$  In particular, for part (1), there are two $q$-integer contributions from reduced subwords of the form
$$tr\underbrace{stst\ldots}_k\quad\text{and}\quad \underbrace{tstst\ldots}_k.$$  For part (2) there are three $q$-integer contributions from reduced subwords of the form
$$r\underbrace{tsts\ldots}_k\quad\text{and}\quad r\underbrace{stst\ldots}_k\quad\text{and}\quad rstr\underbrace{stst\ldots}_k.$$\end{proof}

The polynomials in parts (1) and (3) of the lemma are palindromic, while the
polynomial is part (2) is 3-palindromic but not 4-palindromic.  We now prove the
theorem stated in the introduction.

\subsection*{Proof of Theorem \ref{T:main}.}
Suppose that $W$ avoids all triangles in $\Tria.$  Let $w=v_1v_2\cdots v_{|\supp(w)|}\in W$ be a
complete Grassmannian factorization.  Then by Theorem \ref{T:mainchar} and Proposition
\ref{P:para}, we have that
$$P_w(q)=\prod_{i=1}^{|\supp(w)|} P_{v_i}^{J_i}(q).$$ where
$J_i:=\supp(v_1)\cup\cdots\cup\supp(v_{i-1})$ and $J_1:=\emptyset.$  Moreover, the factors
$P_{v_i}^{J_i}(q)$ are given by either Equation \eqref{E:rank2poly} or by parts (1)--(3) of Lemma
\ref{L:rank3v}.  Now $P_w(q)$ is 4-palindromic if the polynomial in Lemma \ref{L:rank3v} part (2)
does not appear as one of the factors $P_{v_i}^{J_i}(q).$  Since all other possible choices for
$P_{v_i}^{J_i}(q)$ are palindromic, we have that $P_w(q)$ is 4-palindromic if and only
if it is palindromic.  This proves part (1) of Theorem \ref{T:main}.


If $W$ also avoids the triangles of the form $(3,3,c),$ then Lemma \ref{L:rank3v} part (2) is never an option
for $P_{v_i}^{J_i}(q)$.  Hence every $2$-palindromic $w\in W$ is palindromic.  This completes the
proof.$\hfill\Box$

\subsection{Examples}
Consider the Coxeter group $W$ with $S=\{s_1,s_2,s_3,s_4\}$ defined by the Dynkin diagram in Figure \ref{F:dynkin1}.

\begin{figure}[h!]
\begin{center}
  \begin{tikzpicture}[scale=.5]
    \draw[thick] (-2 cm ,0) -- (2 cm ,0);\draw (-1.5 cm, 2 cm) node {4};
    \draw[thick] (-2 cm ,0) -- (0,3.46 cm);
    \draw[thick] (-2 cm ,0) -- (0,1.2 cm);
    \draw[thick] (2 cm ,0) -- (0,3.46 cm);
    \draw[thick] (2 cm ,0) -- (0,1.2 cm);
    \draw[thick] (0,3.46 cm) -- (0,1.2 cm);
    \draw[draw=white, fill=white] (-2 cm ,0) circle (.5 cm) node {$s_2$};
    \draw[draw=white, fill=white] (2 cm ,0) circle (.5 cm) node {$s_3$};
    \draw[draw=white, fill=white] (0,1.2 cm) circle (.5 cm) node {$s_1$};
    \draw[draw=white, fill=white] (0,3.46 cm) circle (.5 cm) node {$s_4$};
  \end{tikzpicture}
\end{center}
\caption{}\label{F:dynkin1}
\end{figure}

Unlabeled edges are assumed to have label $m_{st}=3$ and if there is no edge between $s$ and $t,$
then $m_{st}=2.$  Clearly, $W$ avoids all triangle groups in $\Tria$ and hence we can apply Theorem
\ref{T:mainchar} to compute Poincar\'{e} polynomials.
\begin{example}Let $w=s_1s_2s_1s_3s_2s_1s_3s_2s_1s_4.$  Then $w$ is $2$-palindromic with $|\supp(w)|=|\pred(w)|=4.$
The following is a complete Grassmannian factorization:
$$w=\underbrace{(s_1)}_{v_1}\underbrace{(s_2s_1)}_{v_2}\underbrace{(s_3s_2s_1s_3s_2s_1)}_{v_3}\underbrace{(s_4)}_{v_4}.$$
The corresponding Poincar\'{e} polynomial factorization is
\begin{align*}P_w(q)&=[2]_q[3]_q([7]_q+q^2[3]_q)[2]_q\\ &=(1+q)(1+q+q^2)(1+q+2q^2+2q^3+2q^4+q^5+q^6)(1+q),\end{align*}
so $P_w(q)$ is palindromic.
\end{example}

\begin{example}Let $w=s_2s_4s_2s_4s_1s_2s_4s_1s_2s_4s_2.$  Then $w$ is $2$-palindromic with $|\supp(w)|=|\pred(w)|=3.$
A complete Grassmannian factorization of $w$ is
$$w=\underbrace{(s_2)}_{v_1}\underbrace{(s_4s_2s_4)}_{v_2}\underbrace{(s_1s_2s_4s_1s_2s_4s_2)}_{v_3}.$$
The corresponding Poincar\'{e} polynomial factorization is
\begin{align*}P_w(q)&=[2]_q[4]_q([8]_q+q^2[4]_q+q^4[1]_q)\\
&=(1+q)(1+q+q^2+q^3)(1+q+2q^2+2q^3+3q^4+2q^5+q^6+q^7)\\
\end{align*}
Note that $\{s_1,s_2,s_4\}$ generates the triangle group $(3,3,4).$  Since $v_3=s_1s_2s_4s_1s_2s_4s_2,$ we have that $w$ is 3-palindromic but not 4-palindromic.
\end{example}

An example of a $\Tria$-avoiding Coxeter group $W$ with commuting relations is given by the Dynkin
diagram in Figure \ref{F:dynkin2} below where $p\geq 3.$

\begin{figure}[h!]
\begin{center}
  \begin{tikzpicture}[scale=.5]
    \draw[thick] (-0 cm ,0) -- (-2.5 cm ,1.73 cm);\draw (3.75 cm, 2.3 cm) node{$\infty$}; 
    \draw[thick] (-0 cm ,0) -- (2.5 cm ,1.73 cm);\draw (1.6 cm, 3 cm) node{$\infty$};\draw (1.6 cm, .46 cm) node{$\infty$};
    \draw[thick] (-0 cm ,0) -- (0,3.46 cm);\draw (-1.6 cm, 3 cm) node{$\infty$};\draw (-1.6 cm, .46 cm) node{$\infty$};
    \draw[thick] (0,3.46 cm) -- (-2.5 cm ,1.73 cm);\draw (.6 cm, 1.73 cm) node{$p$};
    \draw[thick] (0,3.46 cm) -- (2.5 cm ,1.73 cm);
    \draw[thick] (5 cm,1.73 cm) -- (2.5 cm ,1.73 cm);
    \draw[draw=white, fill=white] (0 cm ,0) circle (.5 cm) node {$s_5$};
    \draw[draw=white, fill=white] (-2.5 cm ,1.73 cm) circle (.5 cm) node {$s_1$};
    \draw[draw=white, fill=white] (2.5 cm,1.73 cm) circle (.5 cm) node {$s_2$};
    \draw[draw=white, fill=white] (0,3.46 cm) circle (.5 cm) node {$s_4$};
    \draw[draw=white, fill=white] (5 cm,1.73 cm) circle (.5 cm) node {$s_3$};
  \end{tikzpicture}
\end{center}
\caption{}\label{F:dynkin2}
\end{figure}

Observe that $W$ also avoids all triangle groups of the form $(3,3,c).$ Hence every $2$-palindromic element is
palindromic by Theorem \ref{T:main}.  Moreover, every palindromic polynomial factors into a product
of $q$-integers.  We also remark that $W$ is indecomposable with respect to products and free products
of Coxeter groups.

\subsection{Enumeration and description of palindromic elements}

Theorem \ref{T:mainchar} gives a description of the set of palindromic (resp. $2$-palindromic) elements of any $\Tria$-avoiding Coxeter group.
Specifically, the palindromic (resp. $2$-palindromic) elements are those with a certain Grassmannian factorization.  In this section we provide some applications of this idea.
We start by proving a corollary of Theorem \ref{T:mainchar} on the finiteness of the number of palindromic elements for all $\Tria$-avoiding Coxeter groups.

\begin{cor}\label{C:finite}Let $(W,S)$ be a Coxeter group that avoids all triangle groups in $\Tria.$  Then
$W$ has a finite number of palindromic elements if and only if $m_{st}<\infty$ for all $s,t\in S$
and $W$ avoids the triangle group (3,3,3).\end{cor}

\begin{proof}Theorem \ref{T:mainchar} part (3) implies that the triangle group (3,3,3) contains
an infinite number of palindromic elements. Also, if $m_{rs}=\infty,$ then $W_{\{r,s\}}$ is infinite and every element is palindromic.


Let $m_0$ denote the largest value of $m_{st}$ for $s,t\in S.$  Suppose that $W$ avoids $(3,3,3)$
and $m_0<\infty.$ Let $w\in W$ be palindromic with complete Grassmannian factorization
$w=v_1\cdots w_{|\supp(w)|}.$  By Theorem \ref{T:mainchar}, we have that each factor $v_i$ has
length at most $m_0+3,$ so $$\ell(w)<|\supp(w)|(m_0+3)\leq |S|(m_0+3)$$ and hence the
number of palindromic elements in $W$ is finite.\end{proof}
Corollary \ref{C:finite} also holds if palindromic is replaced by $2$-palindromic.


Note that the Grassmannian factorization of an element provided by Theorem \ref{T:mainchar} is not
necessarily unique.  When $m_{st}\geq 3$ for all $s\neq t,$ we give a modified factorization which
does not have this problem.  To state the modified factorization we need the following definition.

\begin{definition}We say a reduced factorization $w=u_1u_2\cdots u_d$ is separable if
$\supp(u_i)\cap\supp(u_j)=\emptyset$ for all $i\neq j.$ If no such non-trivial factorization
exists, then we say that $w$ is inseparable.\end{definition}

Given any complete Grassmannian factorization of a palindromic element $w=v_1\cdots
v_{|\supp(w)|}$, there is a simple method for constructing a separable factorization.  Let
$(i_1,\ldots, i_d)$ denote the subsequence of integers for which $\ell(v_{i_j})=1.$ Then
$w=u_1\cdots u_d$ is a separable factorization where
$$u_j:=v_{i_j}v_{i_j+1}\cdots v_{i_{j+1}-1}$$
and $i_{d+1}:=|\supp(w)|+1.$  We remark that $\ell(v_1)=1$ and hence the sequence $(i_1,\ldots,
i_d)$ is nonempty. Furthermore, each factor $u_j$ is inseparable.  For example, let $W$ be defined
by the Dynkin diagram in Figure \ref{F:dynkin1} and $w=s_4s_2s_4s_2s_3s_1s_3.$ Then $w=u_1u_2$
given by
\begin{equation}\label{E:insep_example}
w=\overbrace{\underbrace{s_4}_{v_1}\underbrace{s_2s_4s_2}_{v_2}}^{u_1}\overbrace{\underbrace{s_3}_{v_3}\underbrace{s_1s_3}_{v_4}}^{u_2}\end{equation}
is a separable factorization.   The following corollary follows from Theorem \ref{T:mainchar}.

\begin{cor}\label{C:separable}
Let $W$ be a Coxeter group with $m_{st}\geq 3$ for all $s\neq t,$ and let $w\in W$ be palindromic.
Then $w$ has a unique separable factorization $w=u_1\cdots u_d$ where each $u_i$ is inseparable and
palindromic. Moreover, any complete Grassmannian factorization $u_i=v_1\cdots v_{|\supp(u_i)|}$ is
unique up to choice of $v_1.$ \end{cor}

\begin{proof}
Any element $w$ has a separable factorization $w=u_1\cdots u_d$ where each
$u_i$ is inseparable.  Since $\supp(u_i)$ is distinct and $W$ has no commuting
braid relations, the factorization is unique.  If $w$ is palindromic, then
every $u_i$ is palindromic since $(u_1\cdots u_i)(u_{i+1})$ is a
BP-decomposition with respect to $J=S\backslash\supp(u_i).$

\smallskip

Let $u_i=v_1\cdots v_{|\supp(u_i)|}$ be a complete Grassmannian factorization,
and let $s_j$ be the unique element of $D_L(v_j)$. Note that $v_1 = s_1$. As
mentioned above, since $u_i$ is inseparable, we must have $|\supp(v_j)| \geq 2$ for
$j = 2, \ldots, |\supp(u_i)|$. Indeed, if $\supp(v_j) = \{s_j\}$ then $s_j$ is the
unique right descent of $v_1 \cdots v_j$, since $s_j \not\in \supp(v_1 \cdots
v_{j-1})$. But by Lemma \ref{L:BPassoc}, $(v_1 \cdots v_j) (v_{j+1} \cdots v_{|\supp(u_i)|})$
is a BP decomposition, so $$\supp(v_{j+1} \dots v_{|\supp(u_i)|}) \cap
\supp(v_1 \cdots v_j) \subset \{s_j\}.$$ Thus $(v_1 \cdots v_{j-1})
(v_{j} \cdots v_{|\supp(u_i)|})$ is a separable factorization, which is a
contradiction.

\smallskip

We now show that $s_j$ is the unique left descent of $v_j \cdots
v_{|\supp(u_i)|}$, for $j \geq 2$. Indeed, looking ahead to Lemma
\ref{L:stdescent}, and using the fact that $|\supp(v_j)| \geq 2$, we see that
$D_L(v_j \cdots v_{|\supp(u_i)|})$ is a subset of $\supp(v_j) \setminus \supp(v_1\cdots
v_{j-1}) = \{s_j\}$. Hence the sequence $(s_2,\ldots,s_{|\supp(u_i)|})$ is
uniquely determined given the choice of $v_1 = s_1$, and the $v_j$'s
are uniquely determined from the corresponding parabolic decomposition.

\end{proof}

Note that there are at most two complete Grassmannian factorizations of each $u_i$ in Corollary \ref{C:separable}.  For example, taking $u_1$ in Equation
\eqref{E:insep_example}, we have
$$u_1=\underbrace{s_4}_{v_1}\underbrace{s_2s_4s_2}_{v_2}=\underbrace{s_2}_{v_1}\underbrace{s_4s_2s_4}_{v_2}$$
as the only two complete Grassmannian factorizations.

Corollary \ref{C:separable} implies that to count the number of palindromic elements of $W,$ it is
sufficient to enumerate elements of $W$ which are inseparable and palindromic.  When $m_{st}$ is
constant we compute an exponential generating series for the number of palindromic elements.
Specifically, let $W(m,n)$ denote the \emph{uniform Coxeter group} such that $|S|=n$ and $m_{st}=m$
for all $s\neq t.$ Uniform Coxeter groups satisfy the property that every $2$-palindromic element
$w$ is palindromic by Theorem \ref{T:mainchar}.  Define the generating series
$$\Phi_m(q,t):=\sum_{n,k\geq 0}P_{n,k}\, \frac{q^kt^n}{n!}$$
where $P_{n,k}$ denotes the number of palindromic $w\in W(m,n)$ of length $k.$ In the case that
$m=2,$ we have $W(2,n)\simeq (\Z/2\Z)^n$ with every element palindromic, so $P_{n,k}=\binom{n}{k}.$  Hence the generating
series
$$\Phi_2(q,t)=\exp(qt+t).$$  For $m\geq 3,$ define
$$\phi_m(q,t):=\sum_{n,k\geq 1}I_{n,k}\, \frac{q^kt^n}{n!}$$ where $I_{n,k}$ denotes the number of palindromic $w\in W(m,n)$ of length $k$ that are inseparable with $|\supp(w)|=n.$
Note that $\Phi_m$ and $\phi_m$ are exponential in $t$ and ordinary in $q$. Corollary
\ref{C:separable} implies
\begin{prop}\label{P:Uniform1}
For any $3\leq m\leq \infty,$ the series
$$\displaystyle\Phi_m(q,t)=\frac{\exp(t)}{1-\phi_m(q,t)}.$$
\end{prop}
The following proposition completes the calculation.
\begin{prop}\label{P:Uniform2}
The exponential generating series for the number of inseparable palindromic elements in $W(m,n)$ is

\begin{equation}\label{E:enumeration}
\phi_m(q,t)=\begin{cases}\quad \displaystyle\frac{(2q-2q^3)t-(3q^3+q^5)t^2}{2-2q^2-4q^2t}\quad\ \
\text{for $m=3$}\\ \\
\displaystyle\frac{2qt-3q^{m}t^2-q^{m+2}[m-3]_qt^3}{2-2q^2t([m-2]_q+q^{m-3})}\quad \text{for $4\leq
m<\infty$}\\ \\ \qquad\qquad \displaystyle\frac{qt-q^2t}{1-q-q^2t}\qquad\qquad\ \ \text{for
$m=\infty$}.
\end{cases}\end{equation}
\end{prop}

\begin{proof}By Theorem \ref{T:mainchar}, $|D_R(w)|\leq 2$ for any palindromic $w\in W(m,n).$  Hence we
can partition the set of inseparable palindromic elements into those with $|D_R(w)|=1,2$
respectively. For notation, let $A_{n,k}$ be the number of inseparable palindromic $w\in W(m,n)$ of
length $k$ with $|\supp(w)|=n$ and $D_R(w)=1.$  Let $B_{n,k}$ be the number of those same elements with $D_R(w)=2.$
We have $I_{n,k}=A_{n,k}+B_{n,k}.$
Consider the polynomials
$$A_n(q):=\frac{1}{n!}\sum_{k\geq 1} A_{n,k}\, q^k\quad \text{and}\quad B_n(q):=\frac{1}{n!}\sum_{k\geq 1} B_{n,k}\, q^k.$$
If $n=1,$ then
$$ A_1(q) = q\qquad\text{and}\qquad B_1(q) = 0.$$
If $3\leq m<\infty,$ then for  $n=2,$ the inseparable elements have the form $s_1s_2s_1\cdots$ or $s_2s_1s_2\cdots$ where the length is at least 3.  There is also a unique longest element $w_0:=\underbrace{s_1s_2\cdots}_m$ with $|D_R(w_0)|=2.$  This gives
$$A_2(q) =q^3[m-3]_q \qquad\text{and}\qquad  B_2(q) =\frac{q^m}{2}.$$
For the remainder of the proof, let $w=v_1\cdots v_{|\supp(w)|}\in W(m,n)$ be a complete
Grassmannian factorization.  We first consider the case when $m=3.$  If $w$ is palindromic and
inseparable, then by Theorem \ref{T:mainchar}, each $v_i$ is a spiral word as in Theorem
\ref{T:mainchar} part 3.  In particular, for each even length, there is a unique $v_i$ of up to $S_3$ permutation symmetry on the generators $\{r,s,t\}.$  Moreover, if $|\supp(w)|\geq 3,$ then $|D_L(w)|=2.$  Thus for all $n\geq 3,$ we have $A_n(q)=0$ and
$$B_n(q)=\left(\frac{2q^2}{1-q^2}\right)B_{n-1}(q)=\frac{q^3}{2}\left(\frac{2q^2}{1-q^2}\right)^{n-2}.$$
Hence $$\phi_3(q,t)=qt+\frac{q^3t^2}{2}+\frac{q^5}{1-q^2}\sum_{n\geq
3}\left(\frac{2q^2}{1-q^2}\right)^{n-3}t^n.$$  This proves the first equation in
\eqref{E:enumeration}.

Now suppose $4\leq m<\infty.$  In this case, if $w$ is palindromic and inseparable, then Theorem
\ref{T:mainchar} implies $|\supp(v_i)|\leq 2.$   Hence each factor $v_i$ has a reduce expression $stst\cdots$ where $t\in D_L(v_1\cdots v_{i-1}).$  In particular, when constructing $w=v_1\cdots v_{|\supp(w)|},$ there are exactly twice as many choices for $v_i$ if $D_L(v_1\cdots v_{i-1})=2$ than if $D_L(v_1\cdots v_{i-1})=1.$  This yields that for $n\geq 3,$ the polynomials $A_n(q)$ and
$B_n(q)$ satisfy the first order recurrence
\begin{align*}A_n(q)&=q^2[m-3]_q\big(A_{n-1}(q)+2B_{n-1}(q)\big)\\
B_n(q)&=q^{m-1}\big(A_{n-1}(q)+2B_{n-1}(q)\big).\end{align*} This implies that
\begin{align*}\left[\begin{matrix}
  A_n(q)   \\
  B_n(q)
 \end{matrix}\right]&=
 \left[\begin{matrix}
  q^2[m-3]_q & 2q^2[m-3]_q  \\
  q^{m-1} & 2q^{m-1}
 \end{matrix}\right]^{n-2}\left[\begin{matrix}
  A_2(q)   \\
  B_2(q)
 \end{matrix}\right]\\ \\ &= q^5[m-2]_q(q^2[m-3]_q+2q^{m-1})^{n-3}\left[\begin{matrix}
  [m-3]_q   \\
  q^{m-3}
 \end{matrix}\right]\\ \\ &= q^7[m-2]_q([m-2]_q+q^{m-2})^{n-3}\left[\begin{matrix}
  [m-3]_q   \\
  q^{m-3}
 \end{matrix}\right].\end{align*}
Thus
\begin{align*}\phi_m(q,t)&=q\, t+\left(q^3[m-3]_q+\frac{q^m}{2}\right)\, t^2+q^7[m-2]^2\sum_{n\geq
3}\left([m-2]_q+q^{m-3}\right)^{n-3}\, t^n\end{align*} which proves the second equation in
\eqref{E:enumeration}.

Finally, we compute the exponential generating series for the uniform Coxeter group $W(\infty,n)$
by taking the limit of $\phi_m$ in the second equation of \eqref{E:enumeration} as $m\rightarrow
\infty.$ This is equivalent to taking $$q^m\rightarrow 0\quad\text{and}\quad [m]_q\rightarrow
\frac{1}{1-q}$$ which yields the third equation in \eqref{E:enumeration}.\end{proof}

The following equations are the first few terms in the Taylor expansion of $\Phi_m(q,t)$ for
$m=3,4,\infty.$  These calculations were computed using the combinat package for Mupad.

\begin{align*}
\Phi_3(q,t)=1&+(1+q)\, t+(1+2q+2q^2+q^3)\, \frac{t^2}{2}\\
&+(1+3q+6q^2+9q^3+6q^4+6q^5+6q^7+O(q^9))\, \frac{t^3}{6}\\
&+(1+4q+12q^2+30q^3+48q^4+60q^5+54q^6+O(q^7))\, \frac{t^4}{24}+O(t^5)
\end{align*}

\begin{align*}
\Phi_4(q,t)=1&+(1+q)\, t+(1+2q+2q^2+2q^3+q^4)\, \frac{t^2}{2}\\
&+(1+3q+6q^2+12q^3+15q^4+12q^5+12q^6+6q^7)\, \frac{t^3}{6}\\
&+(1+4q+12q^2+36q^3+78q^4+120q^5\\
& \qquad + 156q^6+168q^7+150q^8+120q^9+48q^{10})\, \frac{t^4}{24}+O(t^5)
\end{align*}


\begin{align*}
\Phi_{\infty}(q,t)=1&+(1+q)\, t+(1+2q+2q^2+2q^3+2q^4+2q^5+O(q^6))\, \frac{t^2}{2}\\
&+(1+3q+6q^2+12q^3+18q^4+24q^5+O(q^6))\, \frac{t^3}{6}\\
&+(1+4q+12q^2+36q^3+84q^4+156q^5+O(q^6))\, \frac{t^4}{24}+O(t^5)
\end{align*}

\smallskip

By evaluating $\Phi_m(q,t)$ at $q=1,$ we can recover the total number of palindromic elements in
$W(m,n)$.  By Corollary \ref{C:finite}, this value is finite only when $4\leq m<\infty.$  We list
these values for $4\leq m\leq 8$ and $1\leq n\leq 7$ in Figure \ref{F:list_pal}.

\begin{figure}[h!]
\begin{center}
\begin{tabular}{c|ccccccc}
$m\, \diagdown\, n$ \TabTop \TabBot & 1 & 2 & 3 & 4 & 5 & 6 &7\\ \hline
4 \TabTop\TabBot & 2 & 8 & 67 & 893 & 15596 & 330082 & 8165963\\
5 & 2 & 10 & 115 & 2057 & 47356&1314292 & 42584795\\
6 & 2 & 12 & 175 & 3893 & 110436&3768982 & 150113447\\
7 & 2 & 14 & 247 & 6545 & 219956&8884312 & 418725119\\
8 & 2 & 16 & 331 & 10157 & 393916&18351562 & 997538291\\
\end{tabular}\end{center}
\caption{Number of palindromic elements in $W(m,n)$}\label{F:list_pal}
\end{figure}

\section{Properties of triangle groups in $\Tria$}\label{S:HQgroups}

We discuss a few properties of triangle groups in $\Tria.$  The first property is that there are
$k$-palindromic Poincar\'{e} polynomials which are not palindromic for large $k$:

\begin{prop}\label{P:HQprop1}
Let $W$ be the triangle group $(2,b,c)$ with $S=\{r,s,t\}$ such that
$$(rs)^2=(rt)^b=(st)^c=e$$
where $b,c\geq 3$ and $c$ is finite.  Then there exist elements $w\in W$ which
are $(c-2)$-palindromic but not palindromic.\end{prop}

\begin{proof}
Consider $w=uv$ where \begin{equation}\label{E:HQexample} u^{-1}=stst\ldots \qquad \text{and}\qquad
v=rtstst\ldots\end{equation} with $\ell(u)<c$ and $\ell(v)\leq c.$  Calculation of the polynomial
$P_w(q)$ reduces to determining the cardinality of the sets
$$M_k:=\{w'\leq w\ |\ \ell(w')=k\}.$$
First we partition
$$M_k=(M_k\cap W_{\{s,t\}})\, \sqcup\,  (M_k\cap W\backslash W_{\{s,t\}}).$$
If $w'\in W_{\{s,t\}},$ then $w'$ has the form $sts\cdots$ or $tst\cdots$.  Hence
$$|M_k\cap W_{\{s,t\}}|=\begin{cases}
2\qquad\text{if}\quad k<\min\{c,\ell(w)-2\}\\
1\qquad\text{if}\quad k=c\quad\text{or}\quad \ell(w)-1\\
0\qquad\text{if}\quad k>c.
\end{cases}$$
If $w'\in W\backslash W_{\{s,t\}},$ then it is uniquely determined by its parabolic decomposition $w'=u'v'$ where $u'\leq u, v'\leq v$ and $v'$ is nontrivial in $W^{\{s,t\}}.$  Hence $$W\backslash W_{\{s,t\}}\simeq  [e,u]\times([r,v]\cap W^{\{s,t\}}).$$  This gives that
$$|M_k\cap W\backslash W_{\{s,t\}}|=\begin{cases}
2k-1\hspace{.9in}\text{if}\quad k\leq \min\{\ell(u),\ell(v)\}\\
2\ell(v)\hspace{1in}\text{if}\quad \ell(v)<k\leq \ell(u)\\
2\ell(u)\hspace{1in}\text{if}\quad \ell(u)<k\leq \ell(v)\\
2\ell(w)-2k+1\qquad\text{if}\quad k\geq \max\{\ell(u),\ell(v)\}
\end{cases}$$
In the case that $\ell(u)=\ell(v)=c-1,$
we have that
$$P_w(q)=[\ell(w)+1]_q+q^{c+1}+\sum_{k=1}^{c} 2q^{k}[\ell(w)-2k+1]_q$$
In particular, if we write $P_w(q)=\sum a_iq^i,$ then we have that
$$a_i=a_{\ell(w)-i}=2i+1$$ for $i\leq c-3$ and
$$a_{c-2}=2c-1,\qquad\qquad a_{c}=2c-2.$$
Hence $P_w(q)$ is $(c-2)$-palindromic but not palindromic.  For example, if we take $c=4$ and $w=uv=(sts)(rts),$ then
$$[e,w]\cap W_{\{s,t\}}=\{e,s,t,st,ts,tst,sts,stst\}$$ and
$$[e,w]\cap W\backslash W_{\{s,t\}}=[e,u]\cdot ([r,v]\cap W^{\{s,t\}})=\{e,s,t,st,ts,sts\}\cdot\{r,rt,rts\}.$$  In this case, the Poincar\'{e} polynomial $P_w(q)=1+3q+5q^2+7q^3+6q^4+3q^5+q^6$ is 2-palindromic but not 3-palindromic.\end{proof}

It is tempting to conjecture that, for the triangle groups $(2,b,c)$ as in Proposition
\ref{P:HQprop1}, all $(c-1)$-palindromic elements are palindromic.  However for triangle group
$(2,3,5)$ (Coxeter type $H_3$ with $c=5$) there is a unique length 14 element which is
4-palindromic but not palindromic given by $w=t s r t s r t s r t s r t r.$


Theorem \ref{T:mainchar} states that any Grassmannian factorization of a $2$-palindromic element
$w\in W$ is also a BP-decomposition if $W$ avoids triangles in $\Tria$.  This statement is not true
for Coxeter groups which contain triangles in $\Tria$.

\begin{prop}\label{P:HQprop2}
Let $W$ be a Coxeter group.  Then $W$ avoids all triangle groups in $\Tria$ if and only if every
Grassmannian factorization $w=uv$ where $w$ is palindromic is a BP-decomposition.
\end{prop}

\begin{proof}
By Theorem \ref{T:mainchar}, it suffices to show that for triangle groups $(2,b,c)$ as in
Proposition \ref{P:HQprop1} there are Grassmannian factorizations $w=uv$ of palindromic $w$ which
are not BP-decompositions. Consider $w=uv$ as in Equation \eqref{E:HQexample} with $\ell(u)=2$ and
$\ell(v)=c=m_{st}.$  It is easy to check that $w$ is palindromic and that $w=uv$ is a Grassmannian
factorization with respect to $J=\{s,t\}$ but not a BP-decomposition.
\end{proof}


\section{Descent sets of triangle avoiding groups}\label{S:descents}




In this section, we prove several basic properties of Coxeter groups which avoid triangle groups in
$\Tria.$  We begin with a lemma that holds for all Coxeter groups:
\begin{lemma}\label{L:sdescent}
    Let $W$ be a Coxeter group and $u\in W$. If $s \notin D_L(u)$, then $D_L(su) \setminus \{s\}$ consists of the
    elements $t \in D_L(u)$ such that $u$ has a reduced factorization starting
    with a braid $tsts \cdots$ of length $m_{st}-1$. (If $m_{st} = 2$ then this
    braid consists of only one element.)  In other words,
\begin{equation*}
    D_L(su) = \{s\} \cup \{t \in D_L(u) : u = u_0 u_1, u_0 \in W_{\{s,t\}},
        u_1 \in W^{\{s,t\}}, \ell(u_0) = m_{st} - 1\}.
\end{equation*}\end{lemma}
\begin{proof}
    Let $J = D_L(su)$. Then by \cite{BB05}, $W_J$ is a finite Coxeter group and $su$ has a reduced factorization
    beginning with the maximal element $w_0$ of $W_J$. If $t$ is an element of
    $J \setminus \{s\}$, then $m_{st} < \infty$ and $w_0$ has a reduced
    decomposition starting with the longest element of $W_{\{s,t\}}$.
\end{proof}

We now consider Coxeter groups which avoid triangle groups in $\Tria.$

\begin{lemma}\label{L:finpara}
    If $W$ is a Coxeter group which avoids all triangle groups in $\Tria,$ then the
    only finite parabolic subgroups of $W$ are products of rank $2$ Coxeter groups.


In other words, if $J \subset S$ is such that $W_J$ is finite, then $J$ can be
written as a disjoint union $$J=\bigsqcup_{i} J_i,$$ where $|J_i| \leq 2$ for all
$i$, and $m_{st} = 2$ if $s\in J_i$, $t \in J_j$, $i \neq j$.\end{lemma}

\begin{proof}
    Using the classification of finite Coxeter groups, we see that every finite irreducible Coxeter group
    of rank $\geq 3$ contains a triangle group in $\Tria$.
\end{proof}

If $J = D_L(w),$ then $W_J$ is a finite Coxeter group.  In particular, Lemma \ref{L:finpara}
applies to the parabolic subgroups generated by descent sets of $\Tria$-avoiding Coxeter groups.  The following lemma is the main result of this section.

\begin{lemma}\label{L:stdescent}
    Let $(W,S)$ be a Coxeter group which avoids triangle groups in $\Tria$. Let $r,s \in S$ such
    that $3 \leq m_{rs} \leq \infty$, and suppose $u$ is an element of $W$ such
    that $(rs)u$ is a reduced factorization. Then
    \begin{equation*}
        D_L(rsu) \setminus \{r,s\} = \{t \in D_L(u) : m_{rt} = m_{st} = 2\}.
    \end{equation*}
\end{lemma}
\begin{proof}
    The proposition is obviously true if $u = e$. We proceed by induction on
    the length of $u$. Let $J = D_L(s u)$, and write $J = \bigsqcup J_i$ as in Lemma \ref{L:finpara}. We can further assume that if $J_i
    = \{x,y\},$ then $m_{xy}\geq 3$, and that $s \in J_0.$


    Now if $t\in D_L(r s u)\setminus\{r,s\}$, then by Lemma \ref{L:sdescent} we must
    have $m_{rt} < \infty$ and $rsu$ must have a reduced decomposition
    starting the longest element in $W_{\{r,t\}}.$  If $t \notin J_0$
    then $m_{st} = 2$. Since $W$ avoids triangle groups in $\Tria$,  we have that $m_{rt} = 2$ as well.


    This leaves the possibility that $t \in J_0$, in which case $m_{st}\geq 3.$ Once again, since $W$ avoids triangle groups in $\Tria$, we conclude that
    $m_{rt} \geq 3$. Thus $r s u$ has a reduced factorization $r s u = (rtr)u'$,
    where $\ell(u') = \ell(u)-1$. Now $tr u' = su$, so $s \in D_L(tru')$. But by induction, this implies that $m_{ts} = m_{rs} = 2$, which is a contradiction.  Hence $t\notin J_0.$
\end{proof}

\section{Proof of Theorem \ref{T:mainchar}}\label{S:Proof}
We now prove Theorem \ref{T:mainchar}.  The following assumptions are fixed for the remainder of the section. Let $W$ be
a Coxeter group that avoids all triangle groups in $\Tria.$  Let $w\in W$ be $2$-palindromic with a
Grassmannian factorization $w=uv$ with respect to $J=\supp(u).$  By Lemma \ref{L:2palcondition} we
have that
\begin{equation}\label{E:oneguy}|u\cdot\pred(v)\cap\pred(w)|=1.\end{equation}
This implies \begin{equation}\label{E:onepred}|W^J\cap \pred(v)|=1,\end{equation} and in
particular, $|D_R(v)|=1.$  Let $z\in D_R(v)$ denote this unique simple reflection.  The element
$vz$ is the unique element in $W^J\cap \pred(v).$


We divide the proof into three steps.  The first step is to prove that $\supp(v)$ has at most three
elements.  Second, we prove the characterization of $v$ when $\supp(v)$ has exactly three elements.
For the last step, we show that $w=uv$ is BP-decomposition.  We begin with the following technical
lemma.

\begin{lemma}\label{L:prefix}
Let $s_{1},\ldots, s_k$ be the longest sequence of distinct simple reflections such that
$v$ has a reduced decomposition
\begin{equation*}
    v = s_{1} \cdots s_{k} v',
\end{equation*}
and for all $j < k$, $m_{s_j s_{j+1}} \geq 3$. For any $1\leq j\leq k,$ define the set
$I_j:=\{s_{1},\ldots,s_{j}\}.$  Then:

\begin{enumerate}

\item $s_{j}\cdots s_{k}v'\in W^{I_{j-1}}$ for all $j\leq k$, and

\smallskip

\item $\supp(v')\subseteq \{s_{1},\ldots, s_{k}\}.$

\end{enumerate}
\end{lemma}

\begin{proof}Clearly the lemma is true if $\ell(v)=1$ and hence we assume that $\ell(v)\geq 2.$  Observe that $k\geq 2,$ otherwise $v\notin W^J.$
For any $j\leq k,$ let $$v=v_jv_j'$$ be a parabolic decomposition with respect to $W_{I_j}.$  It is easy to see that  $v_1=s_1$ and hence $v_1'=s_{2}\cdots s_{k}v'\in W^{I_{1}}.$

\smallskip

Now let $j\geq 2$ and suppose that $\ell(v_j)> j.$  Then there exists $s\in D_R(v_j)$ such that
$$\supp(v_js)=\supp(v_j).$$  By Lemma \ref{L:stdescent}, we have that the left descent sets
$$D_L(v_jsv_j')=D_L(v_js)\cup\{t\in D_L(v_j')\ |\ m_{ts_i}=2\ \text{for}\ i\leq j\}$$ and
$$D_L(v)=D_L(v_jv_j')=D_L(v_j)\cup\{t\in D_L(v_j')\ |\ m_{ts_i}=2\ \text{for}\ i\leq j\}.$$  Since $\supp(v_js)=\supp(v_j),$ the descent sets above are equal.  Hence $v_jsv_j'\in W^J\cap \pred(v).$  If $j<k,$ then $v_j'\neq e$ and consequently $v_jsv_j'\neq vz$, contradicting Equation \eqref{E:onepred}.  Thus $\ell(v_j)=j$ which implies
that $v_j=s_1\cdots s_j$ and $v_j'=s_{j+1}\cdots s_{k}v'\in W^{I_{j}}.$  This proves part (1) of the lemma.


For part (2), suppose that $\supp(v')\nsubseteq \{s_{1},\ldots, s_{k}\}.$  Then $|\supp(v)|>k$ and
$v_k'\neq e.$  We get that $v_k=s_1\cdots s_k$ and $$D_L(v_k')\cap \{s_1,\cdots,s_k\}=\emptyset$$
since $v_k'\in W^{I_k}.$  By the maximality of $k$, we have that $m_{s_k,r}=2$ for all $r\in D_L(v_k').$  We claim that
$$D_L(s_1\cdots s_{k-1}v')=D_L(v).$$  Indeed, if $k\geq 3,$ this follows from Lemma
\ref{L:stdescent}.  Otherwise, if $k=2,$ then $$m_{s_1s_2}=m_{s_1r}=\infty$$ since $W$ avoids all
triangle groups in $\Tria.$  This proves the claim when $k=2.$  In either case we have that
$$s_1\cdots s_{k-1}v'\in W^J\cap\pred(v)$$ which contradicts Equation \eqref{E:onepred}.  Therefore
$\supp(v')\subseteq \{s_{1},\ldots, s_{k}\}.$
\end{proof}

The following proposition completes the first step of in the proof of Theorem \ref{T:mainchar}.

\begin{prop}\label{P:lessthan3}We have that $|\supp(v)|\leq 3.$  Furthermore, if $|\supp(v)|= 3,$ then $\supp(v)$ generates a triangle group $(a,b,c)$ with $a,b,c\geq 3.$  \end{prop}

\begin{proof}
Suppose $|\supp(v)| \geq 4$ and let $v=s_{1}\cdots s_{k}v'$ as in Lemma \ref{L:prefix}.  We first show
by induction on $j$ that

\begin{enumerate}\item $D_L(s_j \cdots s_k v') = \{s_j\}$

\smallskip

\item $m_{s_i s_j}=2$ for $i\in\{1,\ldots,j-2\}$. \end{enumerate}

Indeed, part (1) is trivial for $j=k$.  Suppose part (1) is true for some $j \leq k$. Now by Lemma
\ref{L:prefix}, $s_1 \cdots s_{j-2} s_{j} \cdots s_k v'$ is reduced, and therefore is not an element of $W^J.$  So by Lemma \ref{L:stdescent}, we have that $s_j \in D_L(s_1 \cdots s_{j-2} s_j \cdots s_k
v').$   Moreover, if $j \geq 4$, then $m_{s_i s_j} = 2$
for all $1 \leq i \leq j-2$. If $j=3$, then $s_1 s_3 \cdots s_k v'$ has a reduced expression
beginning with a braid $s_1 s_3 s_1 \cdots $ of length $m_{s_1 s_3} < \infty$. Since $s_1 \not\in
D_L(s_4 \cdots s_k v')$, we conclude that $m_{s_1 s_3} = 2$. Hence part (2) holds for $j$.


Now suppose part (2) holds for all $j >j_0$. Since $|D_L(v)|=1,$ we have that $s_j\notin
D_L(s_{j_0} \cdots s_k v')$ for any $j>j_0.$  Thus part (1) holds for $j_0$. Hence (1) and (2) hold
for all $j$.


Now part (2), combined with the $\Tria$-avoiding condition, implies that
\begin{equation*}m_{s_i s_j} = \begin{cases} \infty & |i-j| = 1 \\ 2 & |i-j| \geq 2 \end{cases}.
\end{equation*}
In other words, if $|\supp(v)| \geq 4$ then $W_{\supp(v)}$ is defined entirely by commuting relations. We
show that this hypothesis implies that $|\supp(v)| \leq 2$. Indeed, suppose $|\supp(v)| \geq 3$, and let $u
= u_1 u_0$, where $u_1 \in {}^{\supp(v)} W$ and $u_0 \in W_{\supp(v)}$.  Here the set ${}^{\supp(v)} W$ denotes the minimal length representatives of the left cosets  $W/W_{\supp(v)}.$  By Equation \eqref{E:oneguy}, the
product $u_0 s_2 \cdots s_k v'$ must not be reduced. We conclude that $D_R(u_0) \cap D_L(s_2 \cdots
s_k v') = \{s_2\}$ since for any
$s_i,s_j\in \supp(u_0)\cup\supp(v)=\supp(v)$ we have $m_{s_is_j}= 2$ or $\infty.$  Moreover, since $D_L(s_1 s_3 \cdots
s_k v') = \{s_1,s_3\}$, the same argument shows that $s_3 \in D_R(u_0)$. But now we have $\{s_2,s_3\}\subseteq D_R(u_0)$ which implies that the $m_{s_2,s_3}$ is finite.  This contradicts the fact that $m_{s_2,s_3} =\infty$. Hence, $|\supp(v)| \leq 3.$

Finally, if $|\supp(v)|=3,$ then by Lemma \ref{L:prefix}, $m_{s_1s_2},m_{s_2s_3}\geq 3.$  If $m_{s_1s_3}=2,$ then the $\Tria$-avoiding condition implies $m_{s_1s_2}=m_{s_2s_3}=\infty.$ We can now apply the previous argument as above show that $\{s_2,s_3\}\subseteq D_R(u_0)$ and hence $m_{s_2,s_3}$ is finite.  Thus we must have $m_{s_1s_3}\geq 3.$  This completes the proof. \end{proof}


For the next step in the proof of Theorem \ref{T:mainchar}, suppose that $|\supp(v)|=3$ with $\supp(v)=\{r,s,t\}.$  By Proposition \ref{P:lessthan3}, we have $m_{rs},m_{rt},m_{st}\geq 3.$  Consider
the reduced factorization $v=xy$ where $$x^{-1}:=tsrtsr\cdots$$ is the largest spiral word prefix
of $v.$ In other words, we can write \begin{equation}\label{E:xy}v=xy=(\cdots rstrst)\cdot
y.\end{equation} Define $x':=xtst.$  It is easy to see that $\ell(x')=\ell(x)-1$ and that $x'$
equals $x$ with the second to last reflection $s$ removed.  For any $0\leq k\leq \ell(x')$ define a
length $k$ suffix $x_k'$ of $x'$ by
$$x_k':=\underbrace{\cdots rstrt}_k$$

\begin{lemma}\label{L:vchar1} For any $0\leq k\leq \ell(x'),$ the following are true:
\begin{enumerate}\item The product $x_k'y$ is a reduced factorization.

\smallskip

\item

If $k$ is even, then $|D_L(x_k'y)|=1.$  If $k$ is odd, then $|D_L(x_k'y)|\leq 2.$

\smallskip

\item If $|D_L(x_k'y)|=2$ and $k\geq 5,$ then $|D_L(x_{k-2}'y)|=2.$
\end{enumerate}
\end{lemma}

\begin{proof}
If $k=0,$ then $r,t\notin D_L(y)$ since $x$ is a maximal length spiral word.  This implies that
$D_L(y)=\{s\}.$  If $k=1$, then $ty$ is reduced and $D_L(ty)\subseteq\{s,t\}.$  Moreover, $rty$ is
reduced and by Lemma \ref{L:stdescent}, we have that $D_L(rty)=\{r\}$ since $r\notin D_L(y)$. This
proves the lemma for $k\leq 2.$


We proceed with the proof by induction on $k$. Suppose $k \geq 3$. Without loss of generality, we
can assume $r\in D_L(x_k')$, so that $s$ is the first element of $x_{k-1}'$.  We first consider the
case where $k$ is odd.  Then by the inductive assumption, we have that $D_L(x_{k-1}'y)=\{s\}$.
Hence
$$x_k'y=rx_{k-1}'y$$ is reduced and $D_L(x_{k}'y)\subseteq\{r,s\}$.  If $k$ is
even, then $s$ and $t$ are the first two elements of $x_{k-1}'$; in particular,
$r$ is not one of the first two elements. Therefore
$$D_L(x_{k-2}'y)=\{t\}\quad \text{and}\quad D_L(x_{k-1}'y)\subseteq\{s,t\}.$$  So $x_k'y$ is reduced and $D_L(x_k'y)=\{r\}.$
This proves parts (1) and (2) of the lemma.


To prove part (3), suppose that $k\geq 5$ is odd with
$$t\in D_L(x_{k-2}'y)\subseteq \{t,r\}\quad \text{and}\quad r\in D_L(x_k'y)\subseteq\{r,s\}.$$  If $|D_L(x_k'y)|=2,$ then $r\in D_L(x_{k-2}'y)$ since $3\leq m_{rs}<\infty.$   Hence
$|D_L(x_{k-2}'y)|=2.$
\end{proof}

One immediate consequence of Lemma \ref{L:vchar1} is that $x'y$ is a reduced factorization and that
if $\ell(x')$ is even, then $x'y\in W^J\cap \pred(v)$ which is a contradiction to Equation
\eqref{E:onepred}.  Hence $\ell(x')$ is odd (i.e. $\ell(x)$ is even).  The following lemma is a
preliminary characterization of $v.$

\begin{lemma}\label{L:vchar2}
The spiral word $x$ satisfies one of the following conditions:
\begin{enumerate}\item $m_{rt}=3$ and $\ell(x)=4.$
\item $m_{rt}=m_{rs}=3$ and $\ell(x)=6.$
\item $m_{rt}=m_{rs}=m_{st}=3$ and $\ell(x)\geq 8.$
\end{enumerate}
\end{lemma}

\begin{proof}
Since $x'y$ is reduced, we have that $x'y\notin W^J$ and $\ell(x')\geq 3.$   Furthermore, by Lemma
\ref{L:vchar1} part (3), $|D_L(x_k'y)|=2$ for all $k\geq 3.$  In particular the following
statements are true:

\begin{itemize}
\item For $k=3,$ we have that $|D_L(trty)|=2$ if and only if $m_{rt}=3.$

\smallskip

\item For $k=5,$ we have that $|D_L(rstrty)|=2$ if and only if $m_{rt} = m_{rs}=3.$

\smallskip

\item For $k=7,$ we have that $|D_L(strstrty)|=2$ if and only if $m_{rt} = m_{rs} = m_{st}=3.$
\end{itemize}
This completes the proof.
\end{proof}

Now we consider the reduced factorization
\begin{equation}\label{E:vshiftedfactors}
v=xy=(\cdots rstrst)\cdot y=\lefteqn{\overbrace{\phantom{(\cdots rstr)(st}}^x}(\cdots rstr)(
\underbrace{st \ \lefteqn{\overbrace{\phantom{st\cdots)\cdot\bar y }}^y}st\cdots}_{\text{length $k$}} )\cdot\bar y,
\end{equation}
where $k$ is the length of the longest possible prefix of $sty$ the form $stst\cdots$.

\begin{lemma}\label{L:vchar3} With $v$ as Equation \eqref{E:vshiftedfactors}, the following are true:
\begin{enumerate}
\item $\bar y=e.$
\item $k=m_{st}-1$
\end{enumerate}\end{lemma}

\begin{proof}
Suppose that $\bar y\neq e.$  Then $D_L(\bar y)=\{r\}$ by the maximality of $k.$  If $k=2,$ then
$x$ is not a maximal length spiral, and hence $\bar y= e.$
Now assume that $k\geq 3$ and let $v=\bar x\bar z\bar y$ be the reduced factorization given in
\eqref{E:vshiftedfactors} where $\bar z\in W_{\{s,t\}}$ is of length $k.$   Without loss of
generality, let $t\in D_R(\bar z)$ and define $\bar z':=\bar zt.$  Since $k\geq 3,$ we have
$\ell(\bar z')\geq 2$ and thus $\bar x\bar z'\bar y$ is a reduced factorization.  Likewise, since
$\ell(\bar z')\geq 2$ and $D_R(\bar x)=\{r\},$ we have that $\bar x\bar z'\bar y\in W^J$ and hence
$\bar x\bar z'\bar y\in W^J\cap\pred(v).$  But this contradicts Equation \eqref{E:onepred}.
Therefore $\bar y= e$ and part (1) of the lemma is proved.


Since $\ell(x)$ is even, we have that  $k<m_{st},$ otherwise $v\notin W^J.$  This completes the proof in the case of $x$ as in Lemma \ref{L:vchar2} part (3).  Now suppose
that $k\leq m_{st}-2$. If $x$ satisfies the condition in Lemma \ref{L:vchar2} part (1), then
$J=\supp(w)\setminus\{t\}$ and we can write
$$v=tr\underbrace{stst\cdots }_{\text{length $k$}}=tr\bar z.$$
But then $t\bar z\in W^J\cap\pred(v)$ which contradicts Equation \eqref{E:onepred}. If $x$
satisfies the condition in Lemma \ref{L:vchar2} part (2), then $J=\supp(w)\setminus\{r\}$ and
$$v=rstr\underbrace{stst\cdots }_{\text{length $k$}}=rstr\bar z.$$  But then $rst\bar z\in
W^J\cap\pred(v)$ which also contradicts Equation \eqref{E:onepred}.  Hence $k>m_{st}-2$ and part
(2) of the lemma is proved.
\end{proof}

It is easy to see that Lemmas \ref{L:vchar1}, \ref{L:vchar2} and \ref{L:vchar3} prove the
characterization $v$ when $|\supp(v)|=3$ in Theorem \ref{T:mainchar}.


The final step in the proof is to show that $w=uv$ is a BP-decomposition. In this step,
we do not assume that $|S(v)|=3$.

\begin{lemma}\label{L:BPlemma2}
For any $s_0\in\supp(v)\cap J,$ there exists $v''\in W^J$ of length $\ell(v'')=\ell(v)-2$ such that
$s_0v''\in\pred(v).$\end{lemma}

\begin{proof}
If $|\supp(v)|\leq 2,$ then the lemma is obvious.  If $|\supp(v)|=3,$ then we can write $v=xy$ as
in Equation \eqref{E:xy} with the notational change that $$x=rstrst\cdots.$$ In other words, we
let $r,s,t$ denote the first three simple reflections appearing in $x$, rather than the last three.
We want to find
$v''$ for $s_0\in\supp(v)\cap J=\{s,t\}.$  Note that with the change in notation, we have that
$m_{rs}=3.$  Recall the definition of $x'$ given after Equation \eqref{E:xy}.  By Lemma
\ref{L:vchar1} part (1) we have that $x'y$ is reduced and hence $x'y\in \pred(v)\cap W^{\{r,s\}}.$
Thus we have a reduced factorization $$x'y=(srs)y'$$ for some $y'.$ For $s_0=s,$ we set $v''=rsy'.$
Then $v''\in W^J$ since $D_L(v'')=\{r\}.$


We now find a $v''$ for $s_0=t.$  Consider the reduced factorization $$v=(rs)(ty'').$$  Clearly $r\notin D_L(ty''),$ otherwise
$v\notin W^J.$  Hence $rty''\in\pred(v)$ and $rty''\notin W^J.$  This implies that $t\in
D_L(rty'')$ and we can write a reduced factorization $$rty''=(trt)y'''$$ for some $y'''.$  We set
$v''=rty'''$ for $s_0=t.$  Since $D_L(v'')=\{r\},$ we get that $v''\in W^J.$  This completes the
proof.\end{proof}

If $s_0 \in S(v) \cap J$ and $v'' \in W^J$, such that $s_0 v'' \in \pred(v)$, then $s_0 \in
D_R(u)$.  Otherwise $u s_0 v'' \in u\cdot\pred(v) \cap \pred(w)$ which contradicts Equation
\eqref{E:oneguy}. Applying Lemma \ref{L:BPlemma}, we get that $w = uv$ is a BP-decomposition. This
completes the proof of Theorem \ref{T:mainchar}.

\subsection*{Acknowledgements} The second author would like to thank the University of British Columbia for its hospitality.
The authors thank Sara Billey for her useful emails and Jim Carrell and Alex Woo for some helpful
discussions.  The authors would also like to thank the referees for their suggestions on improving the manuscript.  The second author was partially supported by NSF grant DMS-1007255.

\bibliographystyle{plain}

\end{document}